\documentclass[psamsfonts, reqno, 11pt, letterpaper]{amsart}
\usepackage{amsfonts,amsmath, amsthm, amssymb, latexsym, epsfig}

\usepackage[all]{xy}
\usepackage{xspace}
\usepackage{comment}
\usepackage{setspace}
\usepackage{enumerate}
\usepackage[mathscr]{eucal}
\usepackage{xcolor}

\usepackage{bbold}

	\advance \topmargin by -\headheight
	\advance \topmargin by -\headsep

	\evensidemargin \oddsidemargin
	\newcommand{\ftn}[3]{ #1 : #2 \rightarrow #3 }
		\newcommand{\setof}[2]{\ensuremath{\left\{ #1 \: : \: #2 \right\}}}

\newcommand{\cok}{\operatorname{cok}}

	\newcommand{\e}{\ensuremath{\mathfrak{e}}\xspace}
	\newcommand{\ZZ}{\ensuremath{\mathbb{Z}}\xspace}

	\newcommand{\OO}{\mathcal{O}}
	
	\newcommand{\NN}{\ensuremath{\mathbb{N}}\xspace}
	\newcommand{\KKK}{\ensuremath{\mathbb{K}}\xspace}

	\newcommand{\kk}{\ensuremath{\mathit{KK}}\xspace}

	\newcommand{\Hom}{\ensuremath{\operatorname{Hom}}}

	\newcommand{\diag}{\ensuremath{\operatorname{diag}}}
	
	\newcommand{\id}{\ensuremath{\operatorname{id}}}

	\newcommand{\multialg}[1]{\mathcal{M}(#1)\xspace}
	\newcommand{\corona}[1]{\mathcal{Q}(#1)\xspace}

	\newcommand{\cstar}{{$C \sp \ast$}\xspace}

	\newcommand{\R}{\ensuremath{\mathbb{R}}\xspace}
	\newcommand{\N}{\ensuremath{\mathbb{N}}\xspace}
	
	\newcommand{\K}{\ensuremath{\mathbb{K}}\xspace}
	
	\newcommand{\catc}{\mathfrak{C}^{*}\text{-}\mathfrak{alg}}
	\newcommand{\catcn}{\mathfrak{C}^{*}\text{-}\mathfrak{alg}_{ \mathrm{nuc} }}
	\newcommand{\kkc}{\mathfrak{KK}}
	\newcommand{\kkcn}{\mathfrak{KK}_{ \mathrm{nuc} } }

	\theoremstyle{plain}
	\newtheorem{thm}{Theorem}[section]
	\newtheorem{lemma}[thm]{Lemma}
	\newtheorem{theor}[thm]{Theorem}
	\newtheorem{propo}[thm]{Proposition}
	\newtheorem{corol}[thm]{Corollary}

	\newtheorem{remar}[thm]{Remark}	
	\theoremstyle{definition}
	\newtheorem{defin}[thm]{Definition}

	\newtheorem{examp}[thm]{Example}

	\numberwithin{equation}{section}
	\numberwithin{figure}{section}

\begin{document}
	\title[Classifying $C^*$-algebras]{Classifying $C^*$-algebras with both finite and infinite subquotients} 
	\author{S{\o}ren Eilers}
        \address{Department of Mathematical Sciences \\
        University of Copenhagen\\
        Universitetsparken~5 \\
        DK-2100 Copenhagen, Denmark}
        \email{eilers@math.ku.dk }
        \author{Gunnar Restorff}
\address{Faculty of Science and Technology\\University of Faroe 
Islands\\N\'oat\'un 3\\FO-100 T\'orshavn\\Faroe Islands}
\email{gunnarr@setur.fo}

	\author{Efren Ruiz}
        \address{Department of Mathematics\\University of Hawaii,
Hilo\\200 W. Kawili St.\\
Hilo, Hawaii\\
96720-4091 USA}
        \email{ruize@hawaii.edu}
        \date{\today}
	

	\keywords{Classification, Extensions, Graph algebras}
	\subjclass[2000]{Primary: 46L35, 37B10 Secondary: 46M15, 46M18}

	\begin{abstract}
	We give a classification result for a certain class of $C^{*}$-algebras $\mathfrak{A}$ over a finite topological space $X$ in which there exists an open set $U$ of $X$ such that $U$ separates the finite and infinite subquotients of $\mathfrak{A}$.  We apply our results to $C^{*}$-algebras arising from graphs. 
	\end{abstract}
        \maketitle
\section{Introduction}

Just like a finite group, any $C^*$-algebra $\mathfrak A$ with finitely many ideals has a decomposition series
\[
0=\mathfrak I_0\triangleleft \mathfrak I_1\triangleleft \dots \triangleleft \mathfrak I_n=\mathfrak A
\]
in such a way that all subquotients are simple.  As in the group case, the simple subquotients are unique up to permutation of isomorphism classes, but far from determine the isomorphism class of $\mathfrak A$.

If we further assert that all simple subquotients are classifiable by algebraic invariants such as $K$-theory we are naturally lead to the pertinent question of which algebraic invariants, if any, classify all of $\mathfrak A$. This question has previously  been studied, leading to a complete solution when all subquotients are AF (cf.\ \cite{af}), and a partial solution when  all are purely infinite (cf.\ \cite{rmrn:uctkkx}, \cite{beko}, \cite{cp:ckalg}), but in the case when some are of one type and some of another, only sporadic results have been found. It is the purpose of this paper to provide a general framework in which classification of $C^*$-algebras can be proved for a large class of $C^*$-algebras of mixed type.

We are able to do so by combining several recent important developments in classification theory, notably
\begin{itemize}
\item Kirchberg's isomorphism result \cite{kirchpure}
\item The corona factorization property \cite{KucNgCFPdef}
\item The universal coefficient theorem of Meyer and Nest \cite{rmrn:uctkkx}
\end{itemize}
with refinements of our previous work \cite{ERRshift} inspired by an idea of R\o rdam \cite{extpurelyinf}.

As useful as these problems may be to test the borders of our understanding of classification we are driven to this project by one class of examples. A \textbf{graph $C^*$-algebra} has the property that all simple subquotients are either AF or purely infinite, and examples of mixed type occur even for very small graphs. For instance, consider the graphs $\{E_n\}_{n\in\NN}$ given below
\begin{center}
\begin{tabular}{ccc}
$\begin{bmatrix}
0&0&0\\
n&3&0\\
1&1&3
\end{bmatrix}$&&$
\xymatrix{&{\bullet}&\\{\bullet}\ar@{-->}[ur]^-{n}\ar@3{->}@(dl,dr)[]&{\bullet}\ar[u]\ar[l]\ar@3{->}@(dl,dr)[]\\&&}$
\end{tabular}
\end{center}
For any $n>0$, $C^*(E_n)$ decomposes into a linear lattice with simple subquotients
\[
\KKK,\OO_3\otimes\KKK,\OO_3.
\]
The results presented in this paper show that $C^*(E_n)\otimes \KKK\simeq C^*(E_{n+4})\otimes \KKK$ and that
there are three stable isomorphism classes
\[
[C^*(E_1)]=[C^*(E_3)], [C^*(E_2)], [C^*(E_4)]
\]

As in our previous paper \cite{ERRshift} the technical focal point in this work is the general question of when one can deduce from the fact that $\mathfrak A$ and $\mathfrak B$ in the extension
\[
\xymatrix{
  0\ar[r]&{\mathfrak B}\ar[r]&{\mathfrak E}\ar[r]&{\mathfrak A}\ar[r]&0}
\]
are classifiable by $K$-theory, that the same is true for $\mathfrak E$. We shall fix a finite (not necessarily Hausdorff) $T_{0}$ topological space $X$ with a non-trivial open subset $U\subseteq X$ and require that $\mathfrak E$ is a $C^*$-algebra over $X$ -- associating ideals in $\mathfrak E$  with open subsets of $X$ -- in such a way that $\mathfrak B$ is the $C^*$-algebra corresponding to $U$. Assuming then that $\mathfrak A$ and $\mathfrak B$ are  $K\!K$-strongly classifiable by their filtered and ordered $K$-theories over $X\backslash U$ and $U$, respectively, we supply conditions on the extension securing that also $\mathfrak E$ is classifiable by filtered and ordered $K$-theory. Our key technical result to this effect, Theorem \ref{thm:intkkel} below, provides stable isomorphism in this context under, among other things, the provision of fullness of the extension  and $K\! K$-liftability of morphisms of filtered $K$-theory. $K\!K$-liftability follows in many cases by the UCT of Meyer and Nest \cite{rmrn:uctkkx} as generalized by Bentmann and K\"ohler \cite{beko}, and we develop in Section \ref{full} several useful tools to establish fullness.

Although we are confident that Theorem \ref{thm:intkkel} will apply in other settings as well, we restrict ourselves to demonstrating how the results lead to classification of certain graph $C^*$-algebras up to stable isomorphism, generalizing results by the first named author and Tomforde in \cite{semt_classgraphalg}. As a consequence of the results in \cite{semt_classgraphalg}, all graph $C^*$-algebras with exactly one non-trivial ideal are classifiable up to stable isomorphism by the six-term exact sequence in $K$-theory together with the positive cone of the ideal, algebra, and quotient, irrespective of the types of the simple subquotients.  Indicating in a (Hasse) diagram of the ideal lattice a simple AF subquotient with a straight line, and a purely infinite subquotient with a curly line (the zero ideal indicated by ``$\circ$''),  the results in \cite{semt_classgraphalg} solved all the cases
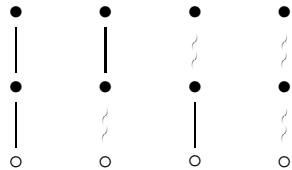
\[
\xymatrix@-0.25cm{{\bullet}\\{\bullet}\ar@{-}[u]\\{\circ}\ar@{-}[u]}
\qquad
\xymatrix@-0.25cm{{\bullet}\\{\bullet}\ar@{-}[u]\\{\circ}\ar@{~}[u]}
\qquad
\xymatrix@-0.25cm{{\bullet}\\{\bullet}\ar@{~}[u]\\{\circ}\ar@{-}[u]}
\qquad
\xymatrix@-0.25cm{{\bullet}\\{\bullet}\ar@{~}[u]\\{\circ}\ar@{~}[u]}
\]
Further, for graph $C^{*}$-algebras with a maximal proper ideal which is AF, indicated diagrammatically as
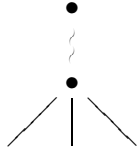
\[
\xymatrix@-0.25cm{
&{\bullet}&\\ 
&{\bullet}\ar@{~}[u]&\\
{}\ar@{-}[ru]&{}\ar@{-}[u]&{}\ar@{-}[lu]
}
\]
are classifiable up to stable isomorphism by the six-term exact sequence in $K$-theory together with the positive cone of the ideal, algebra, and quotient.  In this paper we generalize to the settings
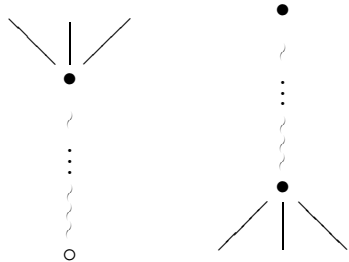
\[
\xymatrix@-0.25cm{\mbox{}&\mbox{}&\mbox{}\\ &{\bullet}\ar@{-}[u]\ar@{-}[ul]\ar@{-}[ur]&\\ & {\vdots}\ar@{~}[u]&\\ &{\circ}\ar@{~}[u]&}
\qquad
\xymatrix@-0.25cm{&{\bullet}&\\&{\vdots}\ar@{~}[u]&\\&{\bullet}\ar@{~}[u]&\\{}\ar@{-}[ru]{}&\ar@{-}[u]{}&\ar@{-}[lu]}
\]
where the ideal lattice at one end (top or bottom) has a finite linear lattice of purely infinite $C^*$-algebras, and at the other has an AF algebra. In particular, we cover all graph $C^*$-algebras of finite linear lattices which has no more than one transition from finite to infinite subquotient, such as $C^*(E_n)$ defined above. We speculate that this condition is not necessary (although it certainly is for our approach) and in \cite{segrer:gclil} give partial evidence for this, under extra conditions on the algebraic nature of the involved $K$-theory.

 \section{Notation and Conventions} 
 
 We first start with some definitions and conventions that will be used throughout the paper.

\subsection{$C^{*}$-algebras over topological spaces} Let $X$ be a topological space and let $\mathbb{O}( X)$ be the set of open subsets of $X$, partially ordered by set inclusion $\subseteq$.  A subset $Y$ of $X$ is called \emph{locally closed} if $Y = U \setminus V$ where $U, V \in \mathbb{O} ( X )$ and $V \subseteq U$.  The set of all locally closed subsets of $X$ will be denoted by $\mathbb{LC}(X)$.  The set of all connected, non-empty locally closed subsets of $X$ will be denoted by $\mathbb{LC}(X)^{*}$.  

The partially ordered set $( \mathbb{O} ( X ) , \subseteq )$ is a \emph{complete lattice}, that is, any subset $S$ of $\mathbb{O} (X)$ has both an infimum $\bigwedge S$ and a supremum $\bigvee S$.  More precisely, for any subset $S$ of $\mathbb{O} ( X )$, 
\begin{equation*}
\bigwedge_{ U \in S } U = \left( \bigcap_{ U \in S } U \right)^{\circ} \quad \mathrm{and} \quad \bigvee_{ U \in S } U = \bigcup_{ U \in S } U
\end{equation*}

For a $C^{*}$-algebra $\mathfrak{A}$, let $\mathbb{I} ( \mathfrak{A} )$ be the set of closed ideals of $\mathfrak{A}$, partially ordered by $\subseteq$.  The partially ordered set $( \mathbb{I} ( \mathfrak{A} ), \subseteq )$ is a complete lattice.  More precisely, for any subset $S$ of $\mathbb{I} ( \mathfrak{A} )$, 
\begin{equation*}
\bigwedge_{ \mathfrak{I} \in S } \mathfrak{I} = \bigcap_{ \mathfrak{I} \in S } \mathfrak{I}  \quad \mathrm{and} \quad \bigvee_{ \mathfrak{I} \in S } \mathfrak{I} = \overline{ \sum_{ \mathfrak{I} \in S } \mathfrak{I} }
\end{equation*}

\begin{defin}
Let $\mathfrak{A}$ be a $C^{*}$-algebra.  Let $\mathrm{Prim} ( \mathfrak{A} )$ denote the \emph{primitive ideal space} of $\mathfrak{A}$, equipped with the usual hull-kernel topology, also called the Jacobson topology.

Let $X$ be a topological space.  A \emph{$C^{*}$-algebra over $X$} is a pair $( \mathfrak{A} , \psi )$ consisting of a $C^{*}$-algebra $\mathfrak{A}$ and a continuous map $\ftn{ \psi }{ \mathrm{Prim} ( \mathfrak{A} ) }{ X }$.  A $C^{*}$-algebra over $X$, $( \mathfrak{A} , \psi )$, is \emph{separable} if $\mathfrak{A}$ is a separable $C^{*}$-algebra.  We say that $( \mathfrak{A} , \psi )$ is \emph{tight} if $\psi$ is a homeomorphism.  
\end{defin}

We always identify $\mathbb{O} ( \mathrm{Prim} ( \mathfrak{A} ) )$ and $\mathbb{I} ( \mathfrak{A} )$ using the lattice isomorphism
\begin{equation*}
U \mapsto \bigcap_{ \mathfrak{p} \in \mathrm{Prim} ( \mathfrak{A} ) \setminus U } \mathfrak{p}
\end{equation*}
 Let $( \mathfrak{A} , \psi )$ be a $C^{*}$-algebra over $X$.  Then we get a map $\ftn{ \psi^{*} }{ \mathbb{O} ( X ) }{ \mathbb{O} ( \mathrm{Prim} ( \mathfrak{A} ) )  \cong \mathbb{I} ( \mathfrak{A} ) }$ defined by
\begin{equation*}
U \mapsto \setof{ \mathfrak{p} \in \mathrm{Prim} ( \mathfrak{A} ) }{ \psi ( \mathfrak{p} ) \in U } = \mathfrak{A}( U )
\end{equation*}
For $Y = U \setminus V \in \mathbb{LC} ( X )$, set $\mathfrak{A}(Y) = \mathfrak{A} (U) / \mathfrak{A}(V)$.   By Lemma 2.15 of \cite{rmrn:bootstrap}, $\mathfrak{A} ( Y)$ does not depend on $U$ and $V$.

\bigskip

In this paper, we will be mainly interested in the following examples:

\begin{examp}
For any $C^{*}$-algebra $\mathfrak{A}$, the pair $( \mathfrak{A} , \id_{ \mathrm{Prim} ( \mathfrak{A} ) } )$ is a tight $C^{*}$-algebra over $\mathrm{Prim} ( \mathfrak{A} )$.  For each $U \in \mathbb{O} ( \mathrm{Prim} ( \mathfrak{A} ) )$, the ideal $\mathfrak{A} ( U )$ equals $\bigcap_{ \mathfrak{p} \in \mathrm{Prim} ( \mathfrak{A} ) \setminus U } \mathfrak{p}$. 
\end{examp}

\begin{examp}\label{Xn}
Let $X_{n} = \{ 1, 2, \dots, n \}$ partially ordered with $\leq$.  Equip $X_{n}$ with the Alexandrov topology, so the non-empty open subsets are 
\begin{equation*}
[a,n ] = \setof{ x \in X }{ a \leq x \leq n}
\end{equation*}
for all $a \in X_{n}$; the non-empty closed subsets are $[1,b]$ with $b \in X_{n}$, and the  non-empty locally closed subsets are those of the form $[a,b]$ with $a, b \in X_{n}$ and $a \leq b$.  Let $( \mathfrak{A} , \phi )$ be a $C^{*}$-algebra over $X_{n}$.  We will use the following notation throughout the paper:  
\begin{equation*}
\mathfrak{A} [k] = \mathfrak{A} ( \{ k \} ),\ \mathfrak{A}[a,b] = \mathfrak{A} ( [ a , b ] ),\ \text{and} \ \mathfrak{A}(i, j] = \mathfrak{A}[i+1, j].
\end{equation*}
Using the above notation we have ideals $\mathfrak{A} [ a, n ]$ such that 
\begin{equation*}
\{ 0 \} \unlhd \mathfrak{A} [ n]  \unlhd \mathfrak{A} [ n -1, n ]  \unlhd \cdots \unlhd \mathfrak{A}  [ 2, n ]  \unlhd \mathfrak{A}  [1,n] = \mathfrak{A}
\end{equation*}  
\end{examp}

\begin{defin}
Let $\mathfrak{A}$ and $\mathfrak{B}$ be $C^{*}$-algebras over $X$.  A homomorphism $\ftn{ \phi }{ \mathfrak{A} }{ \mathfrak{B} }$ is \emph{$X$-equivariant} if $\phi ( \mathfrak{A} (U) ) \subseteq \mathfrak{B} ( U )$ for all $U \in \mathbb{O}(X)$.  Hence, for every $Y = U \setminus V$, $\phi$ induces a homomorphism $\ftn{ \phi_{Y} }{ \mathfrak{A} ( Y ) }{ \mathfrak{B} (Y) }$.  Let $\catc(X)$ be the category whose objects are $C^{*}$-algebras over $X$ and whose morphisms are $X$-equivariant homomorphisms.  
\end{defin}

\begin{remar}
Suppose $\mathfrak{A}$ and $\mathfrak{B}$ are tight $C^{*}$-algebras over $X_{n}$.  Then it is clear that a $*$-homomorphism $\ftn{ \phi }{ \mathfrak{A} }{ \mathfrak{B} }$ is an isomorphism if and only if $\phi$ is an $X_{n}$-equivariant isomorphism.  We will use this fact in Theorem \ref{thm:classgraph}.
\end{remar}

\begin{remar}
Let $\mathfrak{e}_{i} : 0 \to \mathfrak{B}_{i} \to  \mathfrak{E}_{i} \to \mathfrak{A}_{i} \to 0$ be an extension for $i = 1, 2$.  Note that $\mathfrak{E}_{i}$ can be considered as a $C^{*}$-algebra over $X_{2} = \{ 1,2 \}$ by sending $\emptyset$ to the zero ideal, $\{2\}$ to the image of $\mathfrak{B}_{i}$ in $\mathfrak{E}_{i}$, and $\{1 , 2 \}$ to $\mathfrak{E}_{i}$.  Hence, there exists a one-to-one correspondence between $X_{2}$-equivariant homomorphisms $\ftn{ \phi }{ \mathfrak{E}_{1} }{ \mathfrak{E}_{2} }$ and homomorphisms from $\mathfrak{e}_{1}$ and $\mathfrak{e}_{2}$.
\end{remar}

\subsection{Filtered ordered $K$-theory}

\begin{defin}
Let $X$ be a $T_{0}$ topological space and let $\mathfrak{A}$ be a $C^{*}$-algebra over $X$.  For open subsets $U_{1} , U_{2} , U_{3}$ of $X$ with $U_{1} \subseteq U_{2} \subseteq U_{3}$, set $Y_{1} = U_{2} \setminus U_{1}, Y_{2} = U_{3} \setminus U_{1}, Y_{3} = U_{3} \setminus U_{2} \in \mathbb{LC} ( X )$.  Then we have a six term exact sequence
\begin{equation*}
\xymatrix{
K_{0} ( \mathfrak{A} ( Y_{1}  ) ) \ar[r]^{ \iota_{*} } & K_{0} ( \mathfrak{A} ( Y_{2} ) ) \ar[r]^{ \pi_{*} } & K_{0} ( \mathfrak{A} ( Y_{3} ) ) \ar[d]^{\partial_{*}} \\
K_{1} ( \mathfrak{A} ( Y_{3}  ) ) \ar[u]^{ \partial_{*}} & K_{1} ( \mathfrak{A} ( Y_{2} ) ) \ar[l]^{ \pi_{*} } & K_{1} ( \mathfrak{A} ( Y_{1} ) ) \ar[l]^{\iota_{*}}
}
\end{equation*}
The \emph{filtered $K$-theory $\mathrm{FK} _{X} ( \mathfrak{A} )$ of $\mathfrak{A}$} is the collection of all $K$-groups thus occurring and the natural transformations $\{ \iota_{*}, \pi_{*}, \partial_{*} \}$.  The \emph{filtered, ordered $K$-theory $\mathrm{FK} _{X}^{+} ( \mathfrak{A} )$ of $\mathfrak{A}$} is $\mathrm{FK} _{X} ( \mathfrak{A} )$ of $\mathfrak{A}$ together with $K_{0} ( \mathfrak{A} ( Y ) )_{+}$ for all $Y \in \mathbb{LC} ( X )$.

Let $\mathfrak{A}$ and $\mathfrak{B}$ be $C^{*}$-algebras over $X$.  An \emph{isomorphism} $\ftn{\alpha }{ \mathrm{FK}_{X} ( \mathfrak{A} ) }{  \mathrm{FK}_{ X } ( \mathfrak{B} ) }$ is a collection of group isomorphisms
\begin{equation*}
\ftn{\alpha_{Y, *}} { K_{*} ( \mathfrak{A} (Y) ) }{ K_{*} ( \mathfrak{B} (Y) ) }
\end{equation*}
for each $Y  \in \mathbb{LC} ( X )$ preserving all natural transformations. An
\emph{isomorphism} $\ftn{ \alpha }{  \mathrm{FK}_{X}^{+} ( \mathfrak{A} ) }{  \mathrm{FK}_{ X }^{+} ( \mathfrak{B} ) }$ is an isomorphism $\ftn{\alpha }{ \mathrm{FK}_{X} ( \mathfrak{A} ) }{  \mathrm{FK}_{ X } ( \mathfrak{B} ) }$ which satisfies that $\alpha_{Y,0}$ is an order isomorphism for all $Y  \in \mathbb{LC} ( X )$.
\end{defin}

If $Y \in \mathbb{LC}(X)$ such that $Y = Y_{1} \sqcup Y_{2}$ with two disjoint relatively open subsets $Y_{1} , Y_{2} \in \mathbb{O} ( Y ) \subseteq \mathbb{LC} (X)$, then $\mathfrak{A} (Y) \cong \mathfrak{A} ( Y_{1} ) \times \mathfrak{A} ( Y_{2} )$ for any $C^{*}$-algebra over $X$.  Moreover, there is a natural isomorphism $K_{*} ( \mathfrak{A} (Y) )$ to $K_{*} ( \mathfrak{A}( Y_{1} ) ) \oplus K_{*} ( \mathfrak{A} ( Y_{2} ) )$ which is a positive isomorphism from $K_{0} ( \mathfrak{A} (Y) )$ to $K_{0} ( \mathfrak{A}( Y_{1} ) ) \oplus K_{0} ( \mathfrak{A} ( Y_{2} ) )$.  If $X$ is finite, any locally closed subset is a disjoint union of its connected components.  Therefore, we lose no information when we replace $\mathbb{LC} ( X )$ by the subset $\mathbb{LC} ( X )^{*}$.  This observation reduces computation and will be used in Section \ref{examples}.

\section{The $\kk( X_{2}; - , - )$ functor}

Let $a$ be an element of a \cstar-algebra $\mathfrak{A}$.  We say that $a$ is \emph{norm-full in $\mathfrak{A}$} if $a$ is not contained in any norm-closed proper
ideal of $\mathfrak{A}$.  The word ``full'' is also widely used, but since we
will often work in multiplier algebras, we emphasize that it is the
norm topology we are using, rather than the strict topology.  We say that a sub-$C^{*}$-algebra $\mathfrak{B}$ of a $C^{*}$-algebra $\mathfrak{A}$ is \emph{norm-full} if the norm-closed ideal generated by $\mathfrak{B}$ is $\mathfrak{A}$.

\begin{defin}\label{fullext}
An extension $\e$ is said to be \emph{full} if the associated Busby
invariant $\tau_\e$ has the property that $\tau_\e(a)$ is a norm-full
element of $\corona{\mathfrak{B}}  = \multialg{ \mathfrak{B} } / \mathfrak{B}$ for every $a\in \mathfrak{A} \backslash\{0\}$.
\end{defin}

We now define some functors that will be used throughout the rest of the paper.  Let $X$ and $Y$ be $T_{0}$ topological spaces.  For every continuous function $\ftn{f}{X}{Y}$ we have a functor
\begin{equation*}
\ftn{ f }{\catc(X) } { \catc( Y) }, \quad ( A , \psi ) \mapsto (A , f \circ \psi ) 
\end{equation*}  

\begin{itemize}
\item[(1)] Define $\ftn{ g_{ X }^{1} }{ X }{ X_{1} }$ by $g_{X}^{1} ( x )  = 1$.  Then $g_{X}^{1}$ is continuous.  Note that the induced functor $\ftn{ g_{X}^{1} }{ \catc(X) } { \catc( X_{1} ) }$ is the forgetful functor.

\item[(2)]  Let $U$ be an open subset of $X$.  Define $\ftn{ g_{U, X}^{ 2 }}{ X } { X_{2} }$ by $g_{U,X}^{2}(x) = 1$ if $x \notin U$ and $g_{U, X }^{ 2}(x) = 2$ if $x \in U$.  Then $g_{U,X }^{2}$ is continuous.  Thus the induced functor
\begin{equation*}
\ftn{ g_{ U, X}^{2} }{ \catc(X) } { \catc( X_{2} ) }
\end{equation*} 
is just specifying the extension $0 \to \mathfrak{A}(U) \to \mathfrak{A} \to \mathfrak{A} / \mathfrak{A}(U) \to 0$.

\item[(3)] We can generalize (2) to finitely many ideals.  Let $U_{1} \subseteq U_{2} \subseteq \cdots \subseteq U_{n} = X$ be open subsets of $X$.  Define $\ftn{g_{ U_{1} , U_{2} , \dots, U_{n} , X }^{n} }{ X }{ X_{n} }$ by $ g_{ U_{1} , U_{2} , \dots, U_{n} , X }^{n}  (x) = n-k+1$ if $x \in U_{k} \setminus U_{k-1}$.  Then $g_{ U_{1} , U_{2} , \dots, U_{n} , X }^{n}$ is continuous.  Therefore, any $C^{*}$-algebra with ideals $0 \unlhd \mathfrak{I}_{1} \unlhd \mathfrak{I}_{2} \unlhd \cdots \mathfrak{I}_{n} = \mathfrak{A}$ can be made into a $C^{*}$-algebra over $X_{n}$.

\item[(4)]  For all $Y \in \mathbb{LC} ( X )$, $\ftn{ r_{X}^{Y} }{ \catc(X) }{ \catc(Y) }$ is the restriction functor defined in Definition 2.19 of \cite{rmrn:bootstrap}
\end{itemize}
Let $\mathfrak{K}\mathfrak{K}(X)$ be the category whose objects are separable $C^{*}$-algebras over $X$ and the set of morphisms is $\kk( X ; \mathfrak{A} , \mathfrak{B} )$.  By Proposition 3.4 of \cite{rmrn:bootstrap}, these functors induce functors from $\kkc( X )$ to $\kkc ( Z)$, where $Z = Y, X_{1}, X_{n}$.  

\begin{lemma}\label{lem:commfunctor}
Let $U$ be an open set of $X$ and $Y = X \setminus U$.  Then
\begin{align*}
r_{ X_{2} }^{ \{2\} } \circ g_{U,X}^{2} = g_{U}^{1} \circ r_{X}^{U} 
\qquad \mathrm{and} \qquad r_{ X_{2} }^{ \{1\} } \circ g_{U,X}^{2} = g_{ Y }^{1} \circ r_{X}^{Y}
\end{align*}
from $\catc(X)$ to $\catc(X_{1})$.  Consequently, the induced functors from $\mathfrak{K}\mathfrak{K}(X)$ to $\mathfrak{K}\mathfrak{K}(X_{1})$ will be equal.
\end{lemma}

\begin{proof}
Let $\mathfrak{A}$ be a $C^{*}$-algebra over $X$.  Then 
\begin{equation*}
r_{ X_{2} }^{ \{ 2 \} } \circ g_{U,X}^{2} ( \mathfrak{A} ) = \mathfrak{A} ( U ) 
\end{equation*}
and
\begin{equation*}
g_{U}^{1} \circ r_{X}^{U} ( \mathfrak{A} ) = g_{U}^{1} ( \mathfrak{A} ( U ) ) = \mathfrak{A} ( U ).
\end{equation*}

Suppose $\mathfrak{B}$ is a $C^{*}$-algebra over $X$ and $\ftn{ \phi }{ \mathfrak{A} }{ \mathfrak{B} }$ is an $X$-equivariant homomorphism.  Then 
\begin{equation*}
r_{X_{2}}^{ \{2\} } \circ g_{U,X}^{2} ( \phi ) = \phi_{U}
\end{equation*}
and
\begin{equation*}
g_{U}^{1} \circ r_{X}^{U} ( \phi ) = g_{U}^{1} ( \phi_{U} ) = \phi_{U}.
\end{equation*}
Therefore, $r_{ X_{2} }^{ \{2\} } \circ g_{U,X}^{2} = g_{U}^{1} \circ r_{X}^{U}$.

Similar computation shows that $r_{ X_{2} }^{ \{1\} } \circ g_{U,X}^{2} = g_{ Y }^{1} \circ r_{X}^{Y}$.
\end{proof}

The following theorem is a generalization of Theorem 2.3 of \cite{ERRshift}.  This is the key technical theorem of the paper.

\begin{theor}\label{thm:intkkel}
Let $\mathfrak{A}_{1}$ and $\mathfrak{A}_{2}$ be separable, nuclear $C^{*}$-algebras over $X_{2}$.  Suppose $\mathfrak{e}_{i} : \ 0 \to \mathfrak{A}_{i} [ 2 ] \to \mathfrak{A}_{i} \to \mathfrak{A}_{i} [1 ] \to 0$ is an essential extension for $i = 1, 2$.  If $x \in \kk (X_{2}; \mathfrak{A}_{1} , \mathfrak{A}_{2} )$, then 
\begin{equation*}
r_{X_2}^{\{1\}} ( x ) \times [ \tau_{ \mathfrak{e}_{2} } ] = [ \tau_{ \mathfrak{e}_{1} } ] \times r_{X_{2}}^{ \{2\} } ( x )
\end{equation*}
in $\kk^{1} ( \mathfrak{A}_{1} [ 1 ]  , \mathfrak{A}_{2} [ 2 ] )$.
\end{theor}

\begin{proof}
Let $\catcn ( X_{2} )$ be the category whose objects are nuclear, separable $C^{*}$-algebras over $X_{2}$ and morphisms are $X_{2}$-equivariant homomorphisms.  Let $\kkcn ( X_{2} )$ be the category whose objects are nuclear, separable $C^{*}$-algebras over $X_{2}$ and the morphisms from $\mathfrak{A}$ to $\mathfrak{B}$ are the elements of $\kk ( X_{2} ; \mathfrak{A} , \mathfrak{B} )$.  Let $\kkcn$ be the category whose objects are nuclear, separable $C^{*}$-algebras and the morphisms from $\mathfrak{A}$ to $\mathfrak{B}$ are the elements of $\kk ( \mathfrak{A} , \mathfrak{B} )$.  

Let $\mathfrak{A}$ be in $\catcn(X_{2})$ and let $\pi_{ \mathfrak{A} }$ be the natural projection from $\mathfrak{A}$ to $\mathfrak{A}[1]$.  Let $\ftn{ i_{\mathfrak{A}} }{ S \mathfrak{A} }{ C_{ \pi_{\mathfrak{A} } } }$ and $\ftn{ j_{\mathfrak{A}} }{ \mathfrak{A} [2] }{ C_{ \pi_{\mathfrak{A}} } }$ be the natural embeddings, where 
\begin{align*}
C_{ \mathfrak{\pi}_{ \mathfrak{A} } } = \setof{ ( a, f ) \in \mathfrak{A} \oplus C_{0} ( (0,1] , \mathfrak{A}[1] ) }{ \pi_{ \mathfrak{A} } ( a) = f( 1) }
\end{align*} 
is the mapping cone of $\mathfrak{\pi}_{ \mathfrak{A} }$.  Recall that $\kk ( j_{\mathfrak{A}} )$ is invertible in $\kk ( \mathfrak{A}[2] , C_{ \pi_{ \mathfrak{A} } } )$.  Then the isomorphism from $\kk^{1} ( \mathfrak{A} [ 1 ] , \mathfrak{B} [ 2 ] )$ to $\kk ( S \mathfrak{A}[ 1 ] , \mathfrak{B} [ 2 ] )$ sends the class induced by 
\begin{align*}
0 \to \mathfrak{A} [2 ] \to \mathfrak{A} \to \mathfrak{A} [1] \to 0
\end{align*}
to $\kk ( i_{\mathfrak{A}} ) \times \kk ( j_{\mathfrak{A}} )^{-1}$ (see \cite{jcgs:mapcone} and Theorem 19.5.7 of \cite{blackadarB}).  Using this isomorphism from $\kk^{1} ( \mathfrak{A} [ 1 ] , \mathfrak{B} [ 2 ] )$ to $\kk ( S \mathfrak{A}[ 1 ] , \mathfrak{B} [ 2 ] )$, the equation
\begin{equation*}
r_{X_2}^{\{1\}} ( x ) \times [ \tau_{ \mathfrak{e}_{2} } ] = [ \tau_{ \mathfrak{e}_{1} } ] \times r_{X_{2}}^{ \{2\} } ( x )
\end{equation*}
in $\kk^{1} ( \mathfrak{A}_{1} [ 1 ]  , \mathfrak{A}_{2} [ 2 ] )$ becomes
\begin{align*}
S r_{X_2}^{\{1\}} ( x ) \times \kk ( i_{\mathfrak{A}_{2}} ) \times \kk ( j_{\mathfrak{A}_{2}} )^{-1}  = \kk ( i_{\mathfrak{A}_{1}} ) \times \kk ( j_{\mathfrak{A}_{1}} )^{-1} \times r_{X_{2}}^{ \{2\} } ( x )
\end{align*}
in $\kk ( S \mathfrak{A}[ 1 ] , \mathfrak{B} [ 2 ] )$, where $S$ is the natural isomorphism from $\kk ( \mathfrak{A}, \mathfrak{B} )$ to $\kk ( S \mathfrak{A} , S \mathfrak{B} )$.

Let $m = 1$ or $2$.  Define $\ftn{ F_{m} }{ \catcn }{ \kkcn }$ by 
\begin{align*}
F_{m} ( \mathfrak{A} ) =
\begin{cases}
 \mathfrak{A} [ 2 ], &\text{$m =2$} \\
 S \mathfrak{A} [1], &\text{$m=1$}
\end{cases} 
\quad \text{and} \quad
F_{m} ( \phi ) =
\begin{cases}
 \phi_{ \{ 2 \} }, &\text{$m =2$} \\
 S \phi_{ \{ 1 \} }, &\text{$m=1$}
\end{cases} 
\end{align*}
We claim that $\ftn{ \eta }{ F_{1}}{ F_{2} }$ given by $\eta_{ \mathfrak{A} } = \kk ( i_{\mathfrak{A}} ) \times \kk ( j_{ \mathfrak{A} } )^{-1}$ is a natural transformation between the functors $F_{1}$ and $F_{2}$.  Let $\mathfrak{A}$ and $\mathfrak{B}$ be in $\catcn(X_{2})$ and let $\phi$ be an $X_{2}$-equivariant homomorphism.  By the definition of the mapping cone sequence, there exists a homomorphism $\ftn{ \psi }{ C_{ \pi_{ \mathfrak{A} } } }{  C_{ \pi_{ \mathfrak{B} } } }$ such that the diagrams
\begin{align*}
\xymatrix{
0 \ar[r] & S \mathfrak{A}[1] \ar[r]^{ i_{ \mathfrak{A} } } \ar[d]^{ S \phi_{ \{1\} } } & C_{ \pi_{ \mathfrak{A} } } \ar[r] \ar[d]^{\psi} & \mathfrak{A} \ar[r] \ar[d]^{ \phi } & 0 \\ 
0 \ar[r] & S \mathfrak{B}[1] \ar[r]_{ i_{ \mathfrak{B} } } & C_{ \pi_{ \mathfrak{B} } } \ar[r] & \mathfrak{B} \ar[r] & 0
}
\end{align*}
and
\begin{align*}
\xymatrix{
0 \ar[r] & \mathfrak{A}[2] \ar[r]^{ j_{ \mathfrak{A} } } \ar[d]^{ \phi_{ \{2\} } } & C_{ \pi_{ \mathfrak{A} } } \ar[r] \ar[d]^{\psi} & C \mathfrak{A} \ar[r] \ar[d]^{ C \phi } & 0 \\ 
0 \ar[r] & \mathfrak{B}[1] \ar[r]_{ j_{ \mathfrak{B} } } & C_{ \pi_{ \mathfrak{B} } } \ar[r] & C \mathfrak{B} \ar[r] & 0
}
\end{align*}
are commutative.  Thus, 
\begin{align*}
F_{1} ( \phi ) \times \kk ( i_{ \mathfrak{B} } ) \times \kk ( j_{ \mathfrak{B} } )^{-1} &=  \kk ( S \phi_{ \{1\} } ) \times \kk ( i_{ \mathfrak{B} } ) \times \kk ( j_{ \mathfrak{B} } )^{-1} \\
				 &= \kk ( i_{ \mathfrak{A} } ) \times \kk ( \psi ) \times \kk ( j_{ \mathfrak{B} } )^{-1} \\
				&= \kk ( i_{ \mathfrak{A} } ) \times \kk ( j_{ \mathfrak{A} } )^{-1} \times \kk ( \phi_{\{2\}} ). \\
				&= \kk ( i_{ \mathfrak{A} } ) \times \kk ( j_{ \mathfrak{A} } )^{-1}  \times F_{2} ( \phi ). 
\end{align*}
Hence, $\ftn{ \eta }{ F_{1} }{ F_{2} }$ is a natural transformation.  

Since $F_{1}$, $F_{2}$ are stable, split exact, and homotopy invariant functors, by the universal property of $\kk$, we have that $F_{m}$ induces a functor $\ftn{ \overline{F}_{m} }{ \kkcn ( X_{2} ) }{ \kkcn }$ and $\eta$ induces a natural transformation $\ftn{ \overline{ \eta } }{ \overline{F}_{1} }{ \overline{F}_{2} }$.  In particular, for each $x \in \kk ( X_{2} ; \mathfrak{A}_{1} , \mathfrak{A}_{2} )$, we have that 
\begin{align*}
S r_{X_2}^{\{1\}} ( x ) \times \kk ( i_{\mathfrak{A}_{2} } ) \times \kk ( j_{\mathfrak{A}_{2}} )^{-1}  = \kk ( i_{\mathfrak{A}_{1}} ) \times \kk ( j_{\mathfrak{A}_{1}} )^{-1} \times r_{X_{2}}^{ \{2\} } ( x ).
\end{align*}

By the comments made in the second paragraph of the proof, we have that 
\begin{equation*}
r_{X_2}^{\{1\}} ( x ) \times [ \tau_{ \mathfrak{e}_{2} } ] = [ \tau_{ \mathfrak{e}_{1} } ] \times r_{X_{2}}^{ \{2\} } ( x )
\end{equation*}
in $\kk^{1} ( \mathfrak{A}_{1} [ 1 ]  , \mathfrak{A}_{2} [ 2 ] )$.
\end{proof}

\section{Classification}\label{class}

In this section we prove general classification results for several classes of extension algebras.  The result of this section is a generalization of the classification results obtained by the authors in \cite{ERRshift}.

\begin{defin}\label{def:class}
For a $T_{0}$ topological space $X$, we will consider classes $\mathcal{C}_{X}$ of separable, nuclear $C^{*}$-algebras in the bootstrap category of Rosenberg and Schochet $\mathcal{N}$ such that 
\begin{itemize}
\item[(1)] any element in $\mathcal{C}_{X}$ is a $C^{*}$-algebra over $X$;

\item[(2)] if $\mathfrak{A}$ and $\mathfrak{B}$ are in $\mathcal{C}_{X}$ and there exists an invertible element $\alpha$ in $\kk ( X ; \mathfrak{A}, \mathfrak{B} )$ which induces an isomorphism from $\mathrm{FK}_{X}^{+}( \mathfrak{A} )$ to $\mathrm{FK}_{X}^{+}(B)$, then there exists an isomorphism $\ftn{ \phi }{ \mathfrak{A} }{ \mathfrak{B} }$ such that $\kk ( \phi ) = g_{X}^{1}( \alpha )$.
\end{itemize} 
\end{defin}

We concentrate on the following two classes satisfying (1)--(3) above:

\begin{examp}\label{piex}
Let $\mathfrak{A}$ and $\mathfrak{B}$ be separable, nuclear, stable, $\mathcal{O}_{\infty}$-absorbing tight $C^{*}$-algebras over $X$.  Let $\alpha$ be an invertible element in $\kk ( X ; \mathfrak{A} , \mathfrak{B} )$.  By Kirchberg \cite{kirchpure}, there exists an isomorphism $\ftn{ \phi }{ \mathfrak{A} }{ \mathfrak{B} }$ such that $\kk( X ; \phi ) = \alpha$.     Hence, $\kk ( \phi ) = g_{ X }^{1} ( \kk ( X ; \phi ) ) =  \kk ( \alpha )$.  Thus, if $\mathcal{C}_{X}$ is the class of all stable, separable, nuclear $C^{*}$-algebras over $X$ which are $\mathcal{O}_{\infty}$-absorbing in $\mathcal{N}$, then $\mathcal{C}_{X}$ satisfies the properties of Definition \ref{def:class}.
\end{examp}

\begin{examp}\label{afex}
Let $\mathfrak{A}$ and $\mathfrak{B}$ be stable AF algebras.  Let $\alpha$ be an invertible element in $\kk ( X_{1} ; \mathfrak{A} , \mathfrak{B} ) = \kk ( \mathfrak{A} , \mathfrak{B} )$ which induces an isomorphism from $\mathrm{FK}_{X_{1}}^{+} ( \mathfrak{A} ) = ( K_{0} ( \mathfrak{A} ) , K_{0} ( \mathfrak{A} )_{+} )$ to $\mathrm{FK}_{X_{1}}^{+} ( \mathfrak{B} ) = ( K_{0} ( \mathfrak{B} ), K_{0} ( \mathfrak{B})_{+} )$.  Then by the classification of AF algebras \cite{af}, there exists an isomorphism $\ftn{ \phi }{ \mathfrak{A} } { \mathfrak{B} }$ such that $K_{0} ( \phi ) = K_{0} ( \alpha )$.  Since $\kk ( \mathfrak{A} , \mathfrak{B} ) \cong \Hom ( K_{0} ( \mathfrak{A} ), K_{0} ( \mathfrak{B})  )$, we have that $\kk ( \phi ) = \alpha$.  Thus, if $\mathcal{C}_{X_{1}}$ is the class of all stable, AF algebras, then $\mathcal{C}_{X_{1}}$ satisfies the properties of Definition \ref{def:class}.
\end{examp}

\begin{remar}
The condition
\begin{itemize}
\item[(2')] if $\mathfrak{A}$ and $\mathfrak{B}$ are in $\mathcal{C}_{X}$ and there exists an isomorphism $\beta$ from $\mathrm{FK}_{X}^{+}( \mathfrak{A} )$ to $\mathrm{FK}_{X}^{+}(B)$, then there exists an isomorphism $\ftn{ \phi }{ \mathfrak{A} }{ \mathfrak{B} }$ such that $\phi_* = \beta$.
\end{itemize} 
is more closely suited to our purposes, and (1),(2') is true in general in Example \ref{afex}, but not always in Example \ref{piex}, as pointed out in \cite{rmrn:uctkkx}. In fact, there exists a space $X$ with four points such that (3') fails in Example \ref{piex}.
\end{remar}

\begin{lemma}\label{lem:class}
For $i = 1,2$, let $\mathfrak{e}_{i} : \ 0 \to \mathfrak{I}_{i} \to \mathfrak{E}_{i} \to \mathfrak{A}_{i} \to 0$ be non-unital full extensions.  Suppose $\mathfrak{I}_{i}$ is a stable $C^{*}$-algebra satisfying the corona factorization property.  If there exist an isomorphism $\ftn{ \phi_{0} }{ \mathfrak{I}_{1} }{ \mathfrak{I}_{2} }$ and an isomorphism $\ftn{ \phi_{2} }{ \mathfrak{A}_{1} }{ \mathfrak{A}_{2} }$ such that $\kk ( \phi_{2} ) \times [ \tau_{ \mathfrak{e}_{2} } ] = [ \tau_{ \mathfrak{e}_{1} } ] \times \kk ( \phi_{0} )$, then $\mathfrak{E}_{1}$ is isomorphic to $\mathfrak{E}_{2}$.\footnote{This lemma is not correct as stated.  See arXiv:1505.05951 for the corrected version of Lemma 4.5.  As noted in arXiv:1505.05951, the main results of this paper are true verbatim.} 
\end{lemma}

\begin{proof}
Note that $\mathfrak{e}_{1} \cong \mathfrak{e}_{1} \cdot \phi_{0}$ and $\mathfrak{e}_{2} \cong \phi_{2} \cdot \mathfrak{e}_{2}$, where $\mathfrak{e}_{1} \cdot \phi_{0}$ is the push-out of $\mathfrak{e}_{1}$ along $\phi_{0}$ and $\phi_{2} \cdot \mathfrak{e}_{2}$ is the pull-back of $\mathfrak{e}_{2}$ along $\phi_{2}$ (cf. \cite{extpurelyinf}).  Since $[ \tau_{ \mathfrak{e}_{1} \cdot \phi_{0} } ] = [ \tau_{ \mathfrak{e}_{1} } ] \times \kk ( \phi_{0} ) = \kk ( \phi_{2} ) \times [ \tau_{ \mathfrak{e}_{2} } ] = [ \tau_{ \phi_{2} \cdot \mathfrak{e}_{2} } ]$ in $\kk^{1} ( \mathfrak{A}_{1} , \mathfrak{I}_{2} )$, we have that $[ \tau_{ \mathfrak{e}_{1} \cdot \phi_{0} } ] = [ \tau_{ \phi_{2} \cdot \mathfrak{e}_{2} } ]$.  Since $\tau_{ \mathfrak{e}_{1} \cdot \phi_{0} }$ and $\tau_{ \phi_{2} \cdot \mathfrak{e}_{2} }$ are non-unital full extensions and $\mathfrak{I}_{2}$ satisfies the corona factorization property, by Theorem 3.2(2) of \cite{NgCFP}, there exists a unitary $u$ in $\mathcal{M}( \mathfrak{I}_{2} )$ such that $\mathrm{Ad} ( \pi ( u ) ) \circ \tau_{ \mathfrak{e}_{1} \cdot \phi_{0} } = \tau_{ \phi_{2} \cdot \mathfrak{e}_{2} }$.  Hence, $( \mathrm{Ad} ( u ), \mathrm{Ad} (u) , \id_{ \mathfrak{A}_{1} } )$ is an isomorphism between $\mathfrak{e}_{1} \cdot \phi_{0}$ and $\phi_{2} \cdot \mathfrak{e}_{2}$.  Thus, $\mathfrak{E}_{1}$ is isomorphic to $\mathfrak{E}_{2}$.
\end{proof}

We will apply the theorem below to a certain class of $C^{*}$-algebras arising from graphs.  See Proposition \ref{prop:classgraph1}, Corollary \ref{cor:classgraph1}, Proposition \ref{prop:classgraph2},  and Theorem \ref{thm:classgraph}.

\begin{theor}\label{thm:class1}
Let $X$ be a finite topological space and let $U \in \mathbb{O} ( X )$.  Set $Y = X \setminus U \in \mathbb{LC} (X)$.  For $i=1,2$, let $\mathfrak{A}_{i}$ be a $C^{*}$-algebra over $X$ such that $\mathfrak{A}_{i}$ is a stable, separable, nuclear \cstar-algebra and every simple sub-quotient of $\mathfrak{A}_{i}$ is in the bootstrap category $\mathcal{N}$.

Let $\mathcal{C}_{\mathcal{I} , U}$ and $\mathcal{C}_{\mathcal{Q} , Y}$ be classes of $C^{*}$-algebras that satisfy the properties of Definition \ref{def:class}.  Suppose $\mathfrak{A}_{i} (U) $ is a stable $C^{*}$-algebra in $\mathcal{C}_{\mathcal{I} , U}$ and satisfying the corona factorization property, and $\mathfrak{A}_{i} (Y)$ is a stable $C^{*}$-algebra in $\mathcal{C}_{\mathcal{Q} , Y}$.  Suppose
for $i=1,2$
\begin{align*}
\mathfrak{e}_{i} : \ 0 \to \mathfrak{A}_{i}(U) \to \mathfrak{A}_{i} \to \mathfrak{A}_{i} (Y) \to 0 
\end{align*}
are full extensions.  If there exists an isomorphism $\ftn{ \alpha }{ \mathrm{FK}_{X}( \mathfrak{A}_{1} ) }{ \mathrm{FK}_{X}( \mathfrak{A}_{2} ) }$ such that $\ftn{ \alpha_{U} }{ \mathrm{FK}_{U}^{+} ( \mathfrak{A}_{1} (U) ) }{ \mathrm{FK}_{U}^{+} ( \mathfrak{A}_{2} (U) )}$ and $\ftn{ \alpha_{Y} }{ \mathrm{FK}_{Y}^{+} ( \mathfrak{A}_{1} (Y) ) }{ \mathrm{FK}_{Y}^{+} ( \mathfrak{A}_{2} (Y) ) }$ are isomorphisms, and $\alpha$ lifts to an invertible element in $\kk ( X ; \mathfrak{A}_{1} , \mathfrak{A}_{2} )$, then  $\mathfrak{A}_{1} \cong \mathfrak{A}_{2}$. 
\end{theor}

\begin{proof}
Suppose there exists an isomorphism $\ftn{ \alpha }{ \mathrm{FK}_{X}( \mathfrak{A}_{1} ) }{ \mathrm{FK}_{X}( \mathfrak{A}_{2} ) }$ such that $\ftn{ \alpha_{U} }{ \mathrm{FK}_{U}^{+} ( \mathfrak{A}_{1} (U) ) }{ \mathrm{FK}_{U}^{+} ( \mathfrak{A}_{2} (U) ) }$ and $\ftn{ \alpha_{Y} }{ \mathrm{FK}_{Y}^{+} ( \mathfrak{A}_{1} (Y) ) }{ \mathrm{FK}_{Y}^{+} ( \mathfrak{A}_{2} (Y) ) }$ are isomorphisms, and $\alpha$ lifts to an invertible element in $\kk ( X ; \mathfrak{A}_{1} , \mathfrak{A}_{2} )$.  Let $x \in \kk (X; \mathfrak{A}_{1} , \mathfrak{A}_{2} )$ be this lifting.  Then  $r_{X}^{U}(x)$ is an invertible element in $\kk (U;  \mathfrak{A}_{1}(U) , \mathfrak{A}_{2}(U) )$ and $ r_{X}^{Y}(x)$ is an invertible element in $\kk( Y ; \mathfrak{A}_{1} (Y) , \mathfrak{A}_{2} (Y) )$.  Since $\mathfrak{A}_{1} (U)$ and $\mathfrak{A}_{2} (U)$ are in $\mathcal{C}_{\mathcal{I}, U }$ and $\mathfrak{A}_{1} (Y)$ and $\mathfrak{A}_{2} (Y)$ are in $\mathcal{C}_{\mathcal{Q} , Y}$, there exists an isomorphism $\ftn{ \phi_{0} }{ \mathfrak{A}_{1}(U) }{ \mathfrak{A}_{2}(U) }$ which induces $ r_{X}^{U}( x )$ and there exists an isomorphism $\ftn{ \phi_{2} }{ \mathfrak{A}_{1} (Y) }{ \mathfrak{A}_{2} (Y) }$ which induces $r_{X}^{Y}(x)$.  By Theorem \ref{thm:intkkel} and by Lemma \ref{lem:commfunctor} 
\begin{align*}
\kk ( \phi_{2} ) \times [ \tau_{ \mathfrak{e}_{2} } ]  &= ( g_{ Y }^{1} \circ r_{X}^{Y}(x) ) \times [ \tau_{ \mathfrak{e}_{2} } ] = (r_{X_{2}}^{\{1\}} \circ g_{U,X}^{2} (x) ) \times [ \tau_{ \mathfrak{e}_{2} } ] \\
 &= [ \tau_{ \mathfrak{e}_{1} } ] \times  (r_{ X_{2} }^{ \{2\} } \circ g_{U,X}^{2} )(x) = [ \tau_{ \mathfrak{e}_{1} } ] \times (g_{U}^{1} \circ r_{X}^{U}(x ) ) = [ \tau_{ \mathfrak{e}_{1} } ] \times \kk ( \phi_{0} )
\end{align*}
in $\kk^{1} ( \mathfrak{A}_{1} (Y) , \mathfrak{A}_{2}(U) )$.   Since $\mathfrak{e}_{1}$ and $ \mathfrak{e}_{2}$ are full non-unital extensions and $\mathfrak{A}_{i}(U)$ has the corona factorization property, by Lemma \ref{lem:class} we have that $\mathfrak{A}_{1} \cong \mathfrak{A}_{2}$.
\end{proof}

\section{Full Extensions}\label{full}

In this section, we prove that certain extensions arising from graph $C^{*}$-algebras are necessarily full, allowing one to use the results in Section \ref{class}. 

Let $\mathfrak{I}$ be an ideal of a $C^*$-algebra $\mathfrak{B}$.  Set 
\begin{equation*}
\multialg{ \mathfrak{B} ; \mathfrak{I} } = \setof{ x \in \multialg{ \mathfrak{B} } }{ x \mathfrak{B} \subseteq \mathfrak{I} }
\end{equation*}
It is easy to check that $\multialg{ \mathfrak{B} ; \mathfrak{I} }$ is a (norm-closed, two-sided) ideal of $\multialg{ \mathfrak{B} }$. 

\begin{defin}
Let $\{ f_{n} \}$ be an approximate identity consisting of projections for $\K$, where $f_{0} = 0$ and $f_{n} - f_{n-1}$ is a projection of dimension one.  Let $\mathfrak{A}$ be a unital \cstar-algebra and set $e_{n} =1_{ \mathfrak{A} } \otimes f_{n}$.  Note that $\{ e_{n} \}$ is an approximate identity of $\mathfrak{A} \otimes \K$ consisting of projections.

As in \cite{mr_ideal}, we call an element $X \in \multialg{ \mathfrak{A} \otimes \K }$ \emph{diagonal with respect to $\{ e_{n} \}$} if there exists a strictly increasing sequence $\{ \alpha ( n ) \}$ of integers with $\alpha ( 0 ) = 0$ such that 
\begin{equation*}
X ( e_{ \alpha(n) } - e_{ \alpha ( n - 1) } ) - ( e_{ \alpha(n) } - e_{ \alpha ( n - 1) } ) X = 0
\end{equation*}
for all $n \in \N$.  We write $X = \mathrm{diag} ( x_{1} , x_{2} , \dots )$, where
\begin{equation*}
x_{n} = X ( e_{ \alpha(n) } - e_{ \alpha ( n - 1) } )
\end{equation*}

Conversely, if $\{ x_{n} \}$ is a bounded sequence with $x_{n} \in \mathsf{M}_{ k_{n} } ( \mathfrak{A} )$, then upon identifying $\mathsf{M}_{ k_{n} } ( \mathfrak{A} )$ with
\begin{equation*}
( e_{ \alpha (n) } - e_{ \alpha ( n-1) } ) ( \mathfrak{A} \otimes \K ) ( e_{ \alpha (n) } - e_{ \alpha ( n-1) } )
\end{equation*}
for an appropriate $\alpha ( n )$, we have that $X = \mathrm{diag} ( x_{1} , x_{2} , \dots )$ for some $X \in \multialg{ \mathfrak{A} \otimes \K }$.
\end{defin}

Let $\epsilon > 0$.  Define $\ftn{ f_{\epsilon} }{ \R_{+} }{ \R_{+} }$ by 
\begin{equation*}
f_{\epsilon} ( t) =
\begin{cases}
0 	,&\text{if $t \leq \epsilon$} \\
\epsilon^{-1} ( t - \epsilon )   ,&\text{if $\epsilon \leq t \leq 2 \epsilon$} \\
1				,&\text{if $t \geq 2 \epsilon$}
\end{cases}
\end{equation*}

\begin{theor}\label{thm:fullness}
Let $\mathfrak{B}_{0}$ be a unital, $\mathcal{O}_{\infty}$-absorbing \cstar-algebra.  Let $\mathfrak{I}$ be the largest proper non-trivial ideal of $\mathfrak{B} = \mathfrak{B}_{0} \otimes \K$.  If $x \in \multialg{ \mathfrak{B}  }$ such that $x$ is not an element of $\multialg{ \mathfrak{B} ; \mathfrak{I} } + \mathfrak{B}$, then $\mathcal{I}(x) + \mathfrak{B}= \multialg{\mathfrak{B} }$, where $\mathcal{I} (x)$ is the norm-closed ideal of $\multialg{ \mathfrak{B} }$ generated by $x$. 

Consequently, every nonzero element $x \in \corona{ \mathfrak{B} }$ that is not an element of $\multialg{ \mathfrak{B} ; \mathfrak{I} } / \mathfrak{I}$ is norm-full in $\corona{ \mathfrak{B} } $.
\end{theor}

\begin{proof}
First note that by the proof of Theorem 3.2 of \cite{mr_ideal}, $\multialg{ \mathfrak{B} ; \mathfrak{I} }$ is a proper ideal of $\multialg{ \mathfrak{B} }$.  Let $x \in \multialg{ \mathfrak{B}  } \setminus ( \multialg{ \mathfrak{B} ; \mathfrak{I} } + \mathfrak{B})$.  By Proposition 2.8 (i) of \cite{mr_ideal}, we may assume that $x = \diag ( x_{1} , x_{2} , \dots )$ with respect to $\{ e_{n} \}$, i.e.\ there exists a strictly increasing sequence of integers $\{ \alpha ( n ) \}$ with $\alpha ( 0 ) = 0$ such that $x = \sum_{  k = 1}^{\infty} x_{k}$, where $x_{k} \in ( e_{ \alpha ( k ) } - e_{ \alpha ( k - 1 ) } ) \mathfrak{B} ( e_{ \alpha ( k ) } - e_{ \alpha ( k - 1 ) } )$ and the sum converges in the strict topology.

Let $m \in \N$.  Since $x$ is not an element of $\multialg{ \mathfrak{B} ; \mathfrak{I} }$ and $x$ is not an element of $\mathfrak{B}$, there exists $m' \geq m$ such that $\sum_{ k = m }^{m'} x_{k}$ is not an element of $( e_{ \alpha(m') } - e_{ \alpha (m-1) } ) \mathfrak{I} ( e_{ \alpha(m') } - e_{ \alpha (m-1) } )$.   Hence, there exists $\delta > 0$ such that $\sum_{ k = m }^{m'} f_{ \delta } ( x_{k} )$ is not an element of $( e_{ \alpha(m') } - e_{ \alpha (m-1) } ) \mathfrak{I} ( e_{ \alpha(m') } - e_{ \alpha (m-1) } )$.  Therefore, $\sum_{ k = m }^{m'} f_{ \delta } ( x_{k} )$ is norm-full in $( e_{ \alpha(m')} - e_{ \alpha (m-1) } ) \mathfrak{B} (  e_{ \alpha(m')} - e_{ \alpha (m-1) } )$.  By Proposition 2.2 of \cite{HL_fullext}, there exists $z \in ( e_{ \alpha(m')} - e_{ \alpha (m-1) } ) \mathfrak{B} (  e_{ \alpha(m')} - e_{ \alpha (m-1) } )$ such that 
\begin{equation*}
e_{ \alpha(m') }  - e_{ \alpha (m-1) } = z \left( \sum_{ k = m }^{m'} f_{ \delta } ( x_{k} ) \right) z^{*}
\end{equation*}
Therefore,
\begin{equation*}
1_{ \mathfrak{B}_{0} } \leq e_{ \alpha(m') }  - e_{ \alpha (m-1) } = z \left( \sum_{ k = m }^{m'} f_{ \delta } ( x_{k} ) \right) z^{*}
\end{equation*}

By Corollary 2.7 of \cite{mr_ideal}, $\mathcal{I} ( x ) = \multialg{ \mathfrak{B} }$.
\end{proof}

\begin{corol}\label{cor:fullness}
Let $\mathfrak{B}_{0}$ be a unital, $\mathcal{O}_{\infty}$-absorbing \cstar-algebra.  Let $\mathfrak{I}$ be the largest non-trivial proper ideal of $\mathfrak{B} = \mathfrak{B}_{0} \otimes \K$.  Suppose $\mathfrak{B}$ is an ideal of $\mathfrak{A}$ such that $\mathfrak{e}' : 0 \to \mathfrak{B} / \mathfrak{I} \to \mathfrak{A} / \mathfrak{I} \to \mathfrak{A} / \mathfrak{B} \to 0$ is an essential extension.  Then the extension $\mathfrak{e} : 0 \to \mathfrak{B} \to \mathfrak{A} \to \mathfrak{A} / \mathfrak{B} \to 0$ is a full extension.
\end{corol}

\begin{proof}
Note that the canonical projection from $\mathfrak{A}$ to $\mathfrak{A} / \mathfrak{I}$ is an $X_{2}$-equivariant homomorphism. 
Therefore, by Theorem 2.2 of \cite{ELPmorph}, the diagram
\begin{equation*}
\xymatrix{
\mathfrak{A} / \mathfrak{B} \ar[r]^-{ \tau_{ \mathfrak{e} } } \ar[rd]_-{ \tau_{ \mathfrak{e}' } }& \corona{ \mathfrak{B} }   \ar[d] \\
				& \corona{ \mathfrak{B} / \mathfrak{I} }	
}
\end{equation*}
is commutative.

We will first show that $\mathfrak{e}$ is an essential extension.  Suppose $\mathfrak{D}$ is a nonzero ideal of $\mathfrak{A}$.  Note that 
\begin{align*}
 \left( ( \mathfrak{I} + \mathfrak{D} ) / \mathfrak{I} \right) \cap \mathfrak{B} / \mathfrak{I} = 0 \quad \iff \quad ( \mathfrak{I} + \mathfrak{D} ) / \mathfrak{I} = 0 
\end{align*}  
and the second equality occurs exactly when $\mathfrak{D} \subseteq \mathfrak{I}$.  Suppose that $\mathfrak{D} \subseteq \mathfrak{I}$.  Then it is clear that $\mathfrak{D} \cap \mathfrak{B} = \mathfrak{D} \neq 0$.  Suppose $\mathfrak{D}$ is not a subset of $\mathfrak{I}$.  By the above equivalence, $\left( ( \mathfrak{I} + \mathfrak{D} ) / \mathfrak{I} \right) \cap \mathfrak{B} / \mathfrak{I} \neq  0$.  Hence, there exists $x \in \mathfrak{D}$ such that $x \in \mathfrak{B} \setminus \mathfrak{I}$.  Therefore, $x \in \mathfrak{D} \cap \mathfrak{B}$ and $x \neq 0$.  Hence, $\mathfrak{e}$ is an essential extension.

We now prove that $\mathfrak{e}$ is a full extension.  Note that since $\mathfrak{B} / \mathfrak{I}$ is a stable purely infinite simple \cstar-algebra and $\mathfrak{e}'$ is an essential extension, we have that $\mathfrak{e}'$ is a full extension.  Let $a \in \mathfrak{A} / \mathfrak{B}$ be a nonzero element.  Then the ideal generated by $\tau_{ \mathfrak{e}' } ( a )$ in $\corona{ \mathfrak{B} / \mathfrak{I} } $ is $\corona{ \mathfrak{B} / \mathfrak{I} }$.   

Since 
\begin{equation*}
0 \to \multialg{ \mathfrak{B} ; \mathfrak{I} } / \mathfrak{I} \to \corona{ \mathfrak{B} }  \to \corona{ \mathfrak{B} / \mathfrak{I} } \to 0
\end{equation*}
is an exact sequence, by the above commutative diagram, $\tau_{ \mathfrak{e} } ( a )$ is not an element of $\multialg{ \mathfrak{B} ; \mathfrak{I} } / \mathfrak{I}$.  Hence, by Theorem \ref{thm:fullness}, $\tau_{ \mathfrak{e} } ( a )$ is norm-full in $\corona{ \mathfrak{B} }$.
\end{proof}

\begin{propo}\label{prop:fullext}
Let $\mathfrak{A}$ be a \cstar-algebra and let $\mathfrak{I}$ and $\mathfrak{D}$ be ideals of $\mathfrak{A}$ with $\mathfrak{I} \subseteq \mathfrak{D}$.  Suppose $\mathfrak{D} / \mathfrak{I}$ is an essential ideal of $\mathfrak{A} / \mathfrak{I}$.  Then $\mathfrak{e}_{1} : \ 0 \to \mathfrak{I} \to \mathfrak{A} \to \mathfrak{A} / \mathfrak{I} \to 0$ is a full extension if and only if
$\mathfrak{e}_{2} : \ 0 \to \mathfrak{I} \to \mathfrak{D} \to \mathfrak{D} / \mathfrak{I} \to 0$ is a full extension.
\end{propo}

\begin{proof}
Note that the natural embedding $\ftn{ \iota_{ \mathfrak{D} } }{ \mathfrak{D} }{ \mathfrak{A} }$ is an $X_{2}$-equivariant homomorphism with $\iota_{ \mathfrak{I} } = \id_{ \mathfrak{I} }$.  Hence, the following diagram is commutative
\begin{equation*}
\xymatrix{
0 \ar[r] & \mathfrak{I} \ar[r] \ar@{=}[d] & \mathfrak{D} \ar[r] \ar[d] & \mathfrak{D} / \mathfrak{I} \ar[d]^{ \iota_{ \mathfrak{D} / \mathfrak{I } } } \ar[r] & 0 \\
0 \ar[r] & \mathfrak{I} \ar[r] & \mathfrak{A} \ar[r] & \mathfrak{A} / \mathfrak{I} \ar[r] & 0
}
\end{equation*}
Therefore, the diagram
\begin{equation*}
\xymatrix{
\mathfrak{D} / \mathfrak{I} \ar[r]^-{ \tau_{ \mathfrak{e}_{2} } }\ar[d]_{ \iota_{ \mathfrak{D} / \mathfrak{I} } } & \corona{ \mathfrak{I} }  \ar@{=}[d] \\
\mathfrak{A} / \mathfrak{I} \ar[r]_-{ \tau_{ \mathfrak{e}_{1} } } & \corona{ \mathfrak{I} } 
}
\end{equation*}
is commutative.

Suppose $\mathfrak{e}_{1}$ is a full extension.  Then it is clear from the above diagram that $\mathfrak{e}_{2}$ is a full extension.  Suppose that $\mathfrak{e}_{2}$ is a full extension.   Let $a \in \mathfrak{A} / \mathfrak{I}$ be a non-zero element.  Let $I$ be the ideal generated by $a$ in $\mathfrak{A} / \mathfrak{I}$.  Since $\mathfrak{D} / \mathfrak{I}$ is an essential ideal of $\mathfrak{A} / \mathfrak{I}$, there exists a non-zero $b \in \mathfrak{D} / \mathfrak{I}$ such that $\iota_{ \mathfrak{D} / \mathfrak{I} } ( b ) \in I$.  Hence, $\tau_{ \mathfrak{e}_{2} } ( b ) = ( \tau_{ \mathfrak{e}_{1} } \circ \iota_{ \mathfrak{D} / \mathfrak{I } } )( b )$ is norm-full in $\corona{ \mathfrak{I} } $.  Since $( \tau_{ \mathfrak{e}_{1} } \circ \iota_{ \mathfrak{D} / \mathfrak{I} } )( b )$ is in the ideal generated by $\tau_{ \mathfrak{e}_{1}  } ( a )$ in $\corona{ \mathfrak{I} } $, we have that $\tau_{ \mathfrak{e}_{1} } ( a )$ is norm-full in $\corona{ \mathfrak{I} }$.  Thus, $\mathfrak{e}_{1}$ is a full extension.
\end{proof}

\begin{propo}\label{prop:afpifull}
Let $\mathfrak{A}$ be a graph $C^{*}$-algebra satisfying Condition (K).  Suppose $\mathfrak{I}_{1} \unlhd \mathfrak{I}_{2} \unlhd \mathfrak{A}$ such that $\mathfrak{I}_{1}$ is an AF algebra, $\mathfrak{I}_{1}$ is the largest proper ideal of $\mathfrak{I}_{2}$, $\mathfrak{I}_{2} / \mathfrak{I}_{1}$ is purely infinite.  Then $\mathfrak{e} \ : \ 0 \to \mathfrak{I}_{1} \otimes \K  \to \mathfrak{I}_{2} \otimes \K \to \mathfrak{I}_{2} / \mathfrak{I}_{1} \otimes \K \to 0$ is a full extension.
\end{propo}

\begin{proof}
By \cite{ddmt_arbgraph}, $\mathfrak{A} \otimes \K \cong C^{*} (E) \otimes \K$, where $E$ is a graph satisfying Condition (K) and has no breaking vertices.  Hence, by Theorem 3.6 of \cite{bgrs_idealstrucgraph} and Proposition 3.4 of \cite{bgrs_idealstrucgraph}, $\mathfrak{I}_{2} \otimes \K$ is isomorphic to a $C^{*} ( E_{1} ) \otimes \K$ where $E_{1}$ has no breaking vertices and satisfies Condition (K).  Note that $C^{*} ( E_{1} )$ has a largest proper ideal $\mathfrak{D}_{1}$, that $C^{*} ( E_{1} ) / \mathfrak{D}_{1}$ is purely infinite, $\mathfrak{D}_{1}$ is an AF algebra, and $\mathfrak{D}_{1} \otimes \K \cong \mathfrak{I}_{1} \otimes \K$.  Thus, by Proposition 3.10 of \cite{semt_classgraphalg}, there exists a projection $p \in C^{*} ( E_{1} )$ such that $p C^{*} ( E_{1} ) p \otimes \K \cong C^{*} ( E_{1} ) \otimes \K$ and $p \mathfrak{D}_{1} p$ is stable.  Since $p C^{*} ( E_{1} ) p / p \mathfrak{D}_{1}p$ is unital simple $C^{*}$-algebra and 
\begin{align*}
0 \to p \mathfrak{D}_{1} p \to p C^{*} ( E_{1} ) p \to p C^{*} ( E_{1} ) p / p \mathfrak{D}_{1}p \to 0
\end{align*}
is a unital essential extension, the extension is full, cf.~Lemma 1.5 of \cite{ERRshift}.  Since $p \mathfrak{D}_{1} p$ is stable, Proposition 1.6 of \cite{ERRshift} implies that the extension
\begin{equation*}
0 \to p \mathfrak{D}_{1} p  \otimes \K \to p C^{*} ( E_{1} ) p \otimes \K \to \left( p C^{*} ( E_{1} ) p / p \mathfrak{D}_{1}p \right) \otimes \K \to 0
\end{equation*}
is full.  The proposition now follows since the isomorphism between $p C^{*} ( E_{1} ) p \otimes \K$
and $C^{*} ( E_{1} ) p \otimes \K$ maps the ideal $p\mathfrak{D}_1 p\otimes \K$ onto the ideal
$\mathfrak{D}_1\otimes\K$ by Brown's Theorem, cf.~\cite{heralgs}.
\end{proof}

\begin{corol}\label{cor:afpifull}
Let $\mathfrak{A}$ be a graph $C^{*}$-algebra satisfying Condition (K).  Suppose that $\mathfrak{I}$ is an AF algebra such that $\mathfrak{I}$ is an ideal of $\mathfrak{A}$, for all ideals $\mathfrak{J}$ of $\mathfrak{A}$ we have that $\mathfrak{J} \subseteq \mathfrak{I}$ or $\mathfrak{I} \subseteq \mathfrak{J}$, and $\mathfrak{A} / \mathfrak{I}$ is $\mathcal{O}_{\infty}$-absorbing.  Then $\mathfrak{e} : 0 \to \mathfrak{I} \otimes \K \to \mathfrak{A} \otimes \K \to \mathfrak{A} / \mathfrak{I} \otimes \K \to 0$ is a full extension. 
\end{corol}

\begin{proof}
Let $\setof{ \mathfrak{C}_{n} }{ n \in \N }$ be the set of all minimal ideals of $\mathfrak{A} / \mathfrak{I}$ and let $\mathfrak{A}_{n}$ be an ideal of $\mathfrak{A}$ such that $\mathfrak{I} \subseteq \mathfrak{A}_{n}$ and $\mathfrak{A}_{n} / \mathfrak{I} = \mathfrak{C}_{n}$.  

Let $\mathfrak{J}$ be an ideal of $\mathfrak{A}_{n}$.  Then $\mathfrak{J}$ is an ideal of $\mathfrak{A}$.  Hence, $\mathfrak{J} \subseteq \mathfrak{I}$ or $\mathfrak{I} \subseteq \mathfrak{J}$.  Suppose $\mathfrak{I} \subseteq \mathfrak{J}$ but $\mathfrak{I} \neq \mathfrak{J}$.  Then, $\mathfrak{J} / \mathfrak{I} = \mathfrak{A}_{n} / \mathfrak{I} = \mathfrak{C}_{n}$ since $\mathfrak{C}_{n}$ is simple.  Hence, $\mathfrak{I}$ is the largest proper ideal of $\mathfrak{A}_{n}$, $\mathfrak{I}$ is an AF algebra, and $\mathfrak{A}_{n} / \mathfrak{I}$ is purely infinite.  Therefore, by Proposition \ref{prop:afpifull}, $0 \to \mathfrak{I} \otimes \K \to \mathfrak{A}_{n} \otimes \K \to \mathfrak{C}_{n} \otimes \K \to 0$
is a full extension.  

Let $\mathfrak{D} = \overline{ \sum_{ n = 1}^{\infty } \mathfrak{A}_{n} }$.  Then $
\mathfrak{D}$ is an ideal of $\mathfrak{A}$ such that $\mathfrak{D} / \mathfrak{I}$ is an essential ideal of $\mathfrak{A} / \mathfrak{I}$.  Since $\mathfrak{C}_{i} \cap \mathfrak{C}_{j} = \{ 0 \}$ for $i \neq j$, we have that $0 \to \mathfrak{I} \otimes \K \to \mathfrak{D} \otimes \K \to  \mathfrak{D} / \mathfrak{I} \ \otimes \K \to 0$.  is a full extension.  The corollary now follows from  Proposition \ref{prop:fullext}.
\end{proof}

\section{Applications to graph $C^{*}$-algebras}

Recall our definition of  $X_{n}$ from  Example \ref{Xn} above.  We now apply the results of Section \ref{class} and Section \ref{full} to classify a certain class of graph $C^{*}$-algebras that are tight $C^{*}$-algebras over $X_{n}$.    

\begin{propo}\label{prop:cfprop}
Suppose $\mathfrak{A}$ is a $C^{*}$-algebra with finitely many ideals.  If the stabilization of every simple sub-quotient of $\mathfrak{A} \otimes \K$ satisfies the corona factorization property, then $\mathfrak{A} \otimes \K$ satisfies the corona factorization property.  Consequently, any graph $C^{*}$-algebra with finitely many ideals has the corona factorization property.
\end{propo}

\begin{proof}
We will prove the result of the proposition by induction.  If $\mathfrak{A}$ is simple, then by our assumption, $\mathfrak{A} \otimes \K$ has the corona factorization property.  Suppose that the proposition is true for any $C^{*}$-algebra $\mathfrak{B}$ with at most $n$ ideals such that the stabilization of any simple sub-quotient of $\mathfrak{B} \otimes \K$ satisfies the corona factorization property.  

Let $\mathfrak{A}$ be a $C^{*}$-algebra with $n+1$ ideals such that the stabilization of every simple sub-quotient of $\mathfrak{A} \otimes \K$ satisfies the corona factorization property.  Let $\mathfrak{I}$ be a proper non-trivial ideal of $\mathfrak{A} \otimes \K$.  Then $\mathfrak{I}$ and $\mathfrak{A} / \mathfrak{I}$ are $C^{*}$-algebras with at most $n$ ideals such that the stabilization of every simple sub-quotient of $\mathfrak{I}$ and $\mathfrak{A} / \mathfrak{I}$ satisfies the corona factorization property.  Hence, $\mathfrak{I} \otimes \K$ and $\mathfrak{A} / \mathfrak{I} \otimes \K$ satisfy the corona factorization property.  Therefore, by Theorem 3.1(1) of \cite{KucNgReg}, $\mathfrak{A} \otimes \K$ satisfies the corona factorization property.
\end{proof}

The authors have been recently informed by Eduard Ortega that his joint work with Francesc Perera and Mikael R{\o}rdam (see \cite{OPRcp}) implies that the stabilization of any graph $C^{*}$-algebra has the corona factorization property. 
 
\begin{theor}(Meyer-Nest \cite{rmrn:uctkkx})
For the topological space $X_{n}$, if $\mathfrak{A}$ and $\mathfrak{B}$ are separable, nuclear, $C^{*}$-algebras over $X_{n}$ such that $\mathfrak{A} [ k ]$ and $\mathfrak{B} [ k ]$ are in the bootstrap category $\mathcal{N}$, then any isomorphism $\ftn{ \alpha }{ \mathrm{FK}_{X_{n}}( \mathfrak{A} ) }{ \mathrm{FK}_{X_{n}} (B) }$ lifts to an invertible element in $\kk ( X_{n} ; \mathfrak{A} , \mathfrak{B} )$.
\end{theor}

\begin{propo}\label{prop:classgraph1}
Let $\mathfrak{A}_{1}$ and $\mathfrak{A}_{2}$ separable, nuclear, $C^{*}$-algebras over $X_{n}$.  Suppose $\mathfrak{A}_{i} [ 1 ]$ is an AF algebra and $\mathfrak{A}_{i} [2, n] $ is a tight stable $\mathcal{O}_{\infty}$-absorbing $C^{*}$-algebra over $[2, n]$, and $\mathfrak{A}_{i} [ 2 ]$ is an essential ideal of $\mathfrak{ A }_{i} [1 , 2 ]$.  Then $\mathfrak{A}_{1} \otimes \K \cong \mathfrak{A}_{2} \otimes \K$ if and only if there exists an isomorphism $\ftn{ \alpha }{ \mathrm{FK}_{X_{n}} ( \mathfrak{A}_{1} ) } { \mathrm{FK}_{X_{n}} ( \mathfrak{A}_{2} ) }$ such that $\alpha_{ \{ 1 \} }$ is positive. 
\end{propo}

\begin{proof}
Since $\mathfrak{A}_{i}[2,n]$ is a tight $C^{*}$-algebra over $[2,n]$, by Theorem 3.14 of \cite{realrank}, there exists a norm-full projection $p$ in $\mathfrak{A}_{i}[2,n]$.  By Brown's Theorem \cite{heralgs}, $p \mathfrak{A}_{i}[2,n] p \otimes \K \cong \mathfrak{A}_{i}[2,n] \otimes \K$.  Since $\mathfrak{A}_{i}[2,n]$ is an $\mathcal{O}_{\infty}$-absorbing $C^{*}$-algebra, by Corollary~3.1 of \cite{tomswinterZstable}, $p \mathfrak{A}_{i} [2,n] p$ is an $\mathcal{O}_{\infty}$-absorbing $C^{*}$-algebra.  By Corollary \ref{cor:fullness}, 
\begin{align*}
0  \to \mathfrak{A}_{i} [2, n] \otimes \K \to \mathfrak{A}_{i} \otimes \K \to \mathfrak{A}_{i} [ 1 ] \otimes \K \to 0
\end{align*}
is a full extension for $n = 3$.  Suppose $n = 2$.  Then $\mathfrak{A}_{i}[2,2]$ is a purely infinite simple $C^{*}$-algebra, hence $\corona{ \mathfrak{A}_{i} [ 2,2] \otimes \K}$ is a simple $C^{*}$-algebra.  Thus, 
\begin{align*}
0  \to \mathfrak{A}_{i} [2, 2] \otimes \K \to \mathfrak{A}_{i} \otimes \K \to \mathfrak{A}_{i} [ 1 ] \otimes \K \to 0
\end{align*}
is a full extension.  By Proposition \ref{prop:cfprop}, $\mathfrak{A}_{i} [ 2,n]$ has the corona factorization property.  The theorem now follows from Theorem \ref{thm:class1}.
\end{proof}

The following specialization of the result above is in a certain sense dual to Theorem 4.7 of \cite{semt_classgraphalg}.

\begin{corol}\label{cor:classgraph1}
Let $\mathfrak{A}_{1}$ and $\mathfrak{A}_{2}$ be graph $C^{*}$-algebras satisfying Condition (K) and $\mathfrak{A}_{1}$ and $\mathfrak{A}_{2}$ are $C^{*}$-algebras over $X_{2}$.  Suppose $\mathfrak{A}_{i} [ 2 ]$ is $\mathcal{O}_{\infty}$-absorbing, $\mathfrak{A}_{i} [ 2]$ is the smallest ideal of $\mathfrak{A}_{i}$, and $\mathfrak{A}_{i} [ 1 ]$ is an AF algebra.  Then $\mathfrak{A}_{1} \otimes \K \cong \mathfrak{A}_{2} \otimes \K$ if and only if there exists an isomorphism $\ftn{ \alpha }{ \mathrm{FK}_{X_{n}} ( \mathfrak{A}_{1} ) } { \mathrm{FK}_{X_{n}} ( \mathfrak{A}_{2} ) }$ such that $\alpha_{ \{ 1 \} }$ is positive. 
\end{corol}

\begin{propo}\label{prop:classgraph2}
Let $\mathfrak{A}_{1}$ and $\mathfrak{A}_{2}$ be graph $C^{*}$-algebras satisfying Condition (K).  Suppose $\mathfrak{A}_{i}$ is a $C^{*}$-algebra over $X_{n}$ such that $\mathfrak{A}_{i} [ n ]$ is an AF algebra, for every ideal $\mathfrak{I}$ of $\mathfrak{A}_{i}$ we have that $\mathfrak{I} \subseteq \mathfrak{A}_{i} [ n ]$ or $\mathfrak{A}_{i}[ n ] \subseteq \mathfrak{I}$, and $\mathfrak{A}_{i} [1, n-1]$ is a tight, $\mathcal{O}_{ \infty }$-absorbing $C^{*}$-algebra over $[1,n-1]$.  Then $\mathfrak{A}_{1} \otimes \K \cong \mathfrak{A}_{2} \otimes \K$ if and only if there exists an isomorphism $\ftn{ \alpha }{ \mathrm{FK}_{X_{n}} ( \mathfrak{A}_{1} ) } { \mathrm{FK}_{X_{n}} ( \mathfrak{A}_{2} ) }$ such that $\alpha_{ \{ n \} }$ is positive. 
\end{propo}

\begin{proof}
By Corollary \ref{cor:afpifull}, $0 \to \mathfrak{A}_{i} [ n ] \otimes \K \to \mathfrak{A}_{i} \otimes \K \to \mathfrak{A}_{i} [1,n-1]  \otimes \K \to 0$ is a full extension.  By Lemma 3.10 of \cite{ERRshift}, $\mathfrak{A}_{i} [ n ] \otimes \K$ satisfies the corona factorization property.  The result now follows from Theorem \ref{thm:class1}.
\end{proof}

\begin{remar}
Note that the above propositions do not assume that the ideal lattice of $\mathfrak{A}_{i}$ is linear.  For example, Proposition \ref{prop:classgraph2} can be applied to $C^{*}$-algebras $\mathfrak{A}$ that are tight $C^{*}$-algebras over $X = \{ 1, 2, 3, 4 \}$ with $\mathbb{O} (X) = \{ \emptyset \} \cup \setof{ U \subseteq X }{ 4 \in U }$ such that 
\begin{itemize}
\item[(1)] $\mathfrak{A} ( \{ 4 \} )$ is purely infinite

\item[(2)] $\mathfrak{A}( \{ 1, 2, 3 \} )$ is an AF algebra.
\end{itemize} 

In a forthcoming paper, we study graph $C^{*}$-algebras with small ideal structures similar to the $C^{*}$-algebra described above.  Proposition \ref{prop:classgraph1}, Proposition \ref{prop:classgraph2}, and related results will be used to classify these graph $C^{*}$-algebras.    
\end{remar}

\begin{defin}
For each $n \in \N$, we define a class $\mathcal{C}_{n}$ of graph $C^{*}$-algebras as follows:  $\mathfrak{A}$ is in $\mathcal{C}_{n}$ if
\begin{itemize}
\item[(1)] $\mathfrak{A}$ is a graph $C^{*}$-algebra;

\item[(2)] $\mathfrak{A}$ is a tight $C^{*}$-algebra over $X_{n}$; and

\item[(3)] there exists $U \in \mathbb{O} ( X_{n} )$ such that either $\mathfrak{A}(U)$ is an AF algebra and $\mathfrak{A}( X_{n} \setminus U )$ is $\mathcal{O}_{\infty}$-absorbing or $\mathfrak{A} (U)$ is $\mathcal{O}_{\infty}$-absorbing and $\mathfrak{A}( X \setminus U )$ is an AF algebra.
\end{itemize}
\end{defin}
Note that if $C^{*} (E)$ is an element in $\mathcal{C}_{n}$, then by the proof of Lemma 3.1 of \cite{semt_classgraphalg}, $E$ satisfies Condition (K).  

\begin{theor}\label{thm:classgraph}
Let $E_{1}$ and $E_{2}$ be graphs such that $C^{*} ( E_{1} )$ and $C^{*} ( E_{2} )$ are in $\mathcal{C}_{n}$ for some $n \in \N$.  Then the following are equivalent:
\begin{itemize}
\item[(1)] $C^{*} ( E_{1} ) \otimes \K \cong C^{*} ( E_{2} ) \otimes \K$

\item[(2)]  There exists an isomorphism $\ftn{ \alpha}{\mathrm{FK}_{X_{n}}^{+}( C^{*} ( E_{1} ) )}{\mathrm{FK}_{X_{n}}^{+}( C^{*} ( E_{2} ) )}$
\end{itemize}
\end{theor}

\begin{proof}
Suppose there exists an isomorphism $\ftn{ \alpha}{\mathrm{FK}_{X_{n}}^{+}( C^{*} ( E_{1} ) )}{\mathrm{FK}_{X_{n}}^{+}( C^{*} ( E_{2} ) )}$.  Note that by Cuntz \cite{kthypureinf}, if $\mathfrak{A}$ is an $\mathcal{O}_{\infty}$-absorbing with a norm-full projection, then $K_{0} ( \mathfrak{A} ) = K_{0} ( \mathfrak{A} )_{+}$.  Since $K_{0} ( \mathfrak{B} ) \neq K_{0} ( \mathfrak{B} )_{+}$ for any AF algebra, there is no positive isomorphism from the $K_{0}$-group of an AF algebra to the $K_{0}$-group of an $\mathcal{O}_{\infty}$-absorbing $C^{*}$-algebra with a norm-full projection. 

With the above observation, one of the following four cases must happen:

\begin{itemize}
\item[(i)] $C^{*} ( E_{1} )$ and $C^{*} ( E_{2} )$ are AF algebras;
\item[(ii)] $C^{*} ( E_{2} )$ and $C^{*} ( E_{2} )$ are $\mathcal{O}_{ \infty }$-absorbing;
\item[(iii)]  there exists $2 \leq k \leq n$ such that $C^{*} ( E_{i} )[ k , n ]$ is an AF algebra and $C^{*} ( E_{i} )[ 1, k-1]$ is $\mathcal{O}_{ \infty }$-absorbing for $i = 1, 2$;
\item[(iv)] there exists $2 \leq k \leq n$ such that $C^{*} ( E_{i} )[ k , n]$ is $\mathcal{O}_{\infty}$-absorbing  and $C^{*} ( E_{i} )[1 , k-1 ]$ is an AF algebra for $i = 1, 2$.
\end{itemize}

Case (i) follows from the classification of AF algebras.  Case (ii) follows from Theorem 4.14 of \cite{rmrn:uctkkx}.  Case (iii) follows from Proposition \ref{prop:classgraph2} and Case (iv) follows from Proposition \ref{prop:classgraph1}.
\end{proof}

\section{Examples}\label{examples}

\subsection{Case I}
Fix a prime $p$ and consider the class of graph $C^*$-algebras given by adjacency matrices
\[
\begin{bmatrix}
0&0&0\\
z&p+1&0\\
y&x&p+1
\end{bmatrix}
\]
for $y, z > 0$.  Theorem \ref{thm:classgraph} applies directly as the resulting graph $C^*$-algebra has a finite linear ideal lattice $0\triangleleft \mathfrak{I}_1\triangleleft \mathfrak{I}_2\triangleleft \mathfrak{A}$ with subquotients $\mathfrak{I}_1=\KKK, \mathfrak{I}_2/ \mathfrak{I}_1=\OO_{p+1}\otimes\KKK,$ and $\mathfrak{A} / \mathfrak{I}_2=\OO_{p+1}$. All $K_1$-groups in the filtered $K$-theory vanish, and the $K_0$-groups and the natural transformations
\[
\xymatrix{
{K_0( \mathfrak{I}_1)}\ar[r]\ar@{=}[d]&{K_0( \mathfrak{I}_2)}\ar[r]\ar[d]&{K_0( \mathfrak{I}_2 / \mathfrak{I}_1)}\ar[d]\\
{K_0( \mathfrak{I}_1 ) }\ar[r]&{K_0( \mathfrak{A} ) }\ar[r]\ar[d]&{K_0( \mathfrak{A} / \mathfrak{I}_1)}\ar[d]\\
&{K_0( \mathfrak{A} / \mathfrak{I}_2 )}\ar@{=}[r]&{K_0( \mathfrak{A} / \mathfrak{I}_2)}
}
\]
may be computed as
\[
{\xymatrix{
{\ZZ}\save [].[d].[dr].[r]*[F.]\frm{}\restore \ar[r]\ar@{=}[d]
&{\cok\left[\begin{smallmatrix}z\\p\end{smallmatrix}\right]}\ar[r]\ar[d]&{\cok\begin{bmatrix}p
  \end{bmatrix}}\ar[d]\\
{\ZZ}\ar[r]&{\cok\left[\begin{smallmatrix}z&y\\x&p\\p&0\end{smallmatrix}\right]}\ar[r]\ar[d]&{\cok\left[\begin{smallmatrix}x&p\\p&0
  \end{smallmatrix}\right]}\ar[d]\\
&{\cok\begin{bmatrix}p
  \end{bmatrix}}\ar@{=}[r]&{\cok\begin{bmatrix}p
  \end{bmatrix}}
}}
\]
with all maps induced by the canonical maps from $\ZZ^r$ into $\ZZ^s$ for suitably chosen $r$ and $s$.

Checking when two such filtered $K$-theories are isomorphic is not easy. Of course it would be necessary that
\[
p\mid x\Leftrightarrow p\mid x'\qquad
p\mid z\Leftrightarrow p\mid z'
\]
but depending upon the invertibility of $x$ and $z$ in $\ZZ/p$ we get varying conditions on $y$. The work in \cite{abk} (see \cite[p. 33, 39]{sea}) explains how to reduce this task to checking isomorphism only in the part of the invariant enclosed in dashed lines. With this, it is easy to conclude:

\begin{examp}
With graphs $E$ and $E'$ given by matrices
\[
\begin{bmatrix} 
0&0&0\\
z&p+1&0\\
y&x&p+1
\end{bmatrix}
\qquad
\begin{bmatrix}
0&0&0\\
z'&p+1&0\\
y'&x'&p+1
\end{bmatrix},
\]
respectively, we have
$C^*(E)\otimes \KKK\simeq C^*(E')\otimes \KKK$ precisely when
\begin{enumerate}
\item $p\mid x\Leftrightarrow p\mid x'$, and
\item $p\mid z\Leftrightarrow p\mid z'$, and
\item
\begin{enumerate}
\item $p\mid y\Leftrightarrow p\mid y'$ when $p\mid x$ and $p\mid z$
\item $p\mid [y-xz/p]\Leftrightarrow p\mid [y'-x'z'/p]$ when $p\nmid x$ and $p\mid z$
\end{enumerate}
\end{enumerate}
\end{examp}

\subsection{Case II} 
We now consider graphs given by
\[
\begin{bmatrix} 
0&0&0&0\\ 
x&p+1&0&0\\ 
y&0&p+1&0\\ 
z&0&0&p+1
\end{bmatrix}
\] 
with $x, y, z > 0$.  The resulting ideal lattice is not linear; in fact we have an extension
\begin{equation}\label{ourext}
\xymatrix{{0}\ar[r]&{\KKK}\ar[r]&\mathfrak{A}\ar[r]&{\OO_{p+1}\oplus \OO_{p+1}\oplus \OO_{p+1}}\ar[r]&0
}
\end{equation}
showing that the ideal lattice is precisely of the type demonstrated by Meyer and Nest in \cite{rmrn:uctkkx} to not generally allow a UCT for filtered $K$-theory. But since our $C^*$-algebras have real rank zero, we may appeal to \cite{arr} to see that isomorphisms of the filtered $K$-theory, which in this case has the form 
\[
{\xymatrix{
{}\save [].[dd].[ddr].[r]*[F.]\frm{}\restore&{\cok\left[\begin{smallmatrix}x\\p\end{smallmatrix}\right]}\ar[r]\ar[dr]&{\cok\left[\begin{smallmatrix}x&y\\p&0\\0&p\end{smallmatrix}\right]}\ar[dr]&&{\cok\begin{bmatrix}
p
\end{bmatrix}}\\
{\ZZ}\ar[ur]\ar[r]\ar[dr]&{\cok\left[\begin{smallmatrix}y\\p\end{smallmatrix}\right]}\ar[ur]\ar[dr]&{\cok\left[\begin{smallmatrix}x&z\\p&0\\0&p\end{smallmatrix}\right]}\ar[r]&{\cok\left[\begin{smallmatrix}x&y&z\\p&0&0\\0&p&0\\0&0&p\end{smallmatrix}\right]}\ar[r]\ar[ur]\ar[dr]&{\cok\begin{bmatrix}
p
\end{bmatrix}}\\
&{\cok\left[\begin{smallmatrix}z\\p\end{smallmatrix}\right]}\ar[r]\ar[ur]&{\cok\left[\begin{smallmatrix}y&z\\p&0\\0&p\end{smallmatrix}\right]}\ar[ur]&&{\cok\begin{bmatrix}
p,
\end{bmatrix}}}}
\]
lift to invertible elements of $\kk$, so since the ideal $\mathfrak{I}_1 \cong \KKK$ is a least ideal with $\mathfrak{A}/\mathfrak{I}_1$ absorbing $\OO_\infty$ we may apply Theorem \ref{thm:class1}.
Further, as explained in \cite[p.\ 33, 39]{sea} it suffices by  \cite{abk} to check the existence of isomorphisms on the part of the invariant enclosed in dashed lines, and then 
it is straightforward to determine when the filtered $K$-theory for two such matrices are the same; indeed this amounts to
\[
p\mid x\Leftrightarrow p\mid x'\qquad
p\mid y\Leftrightarrow p\mid y'\qquad
p\mid z\Leftrightarrow p\mid z'.
\]
Taking into account the homeomorphims of $\operatorname{Prim}(\mathfrak A)$ we arrive at

\begin{examp}
With graphs $E$ and $E'$ given by matrices
\[
\begin{bmatrix}
0&0&0&0\\
x&p+1&0&0\\
y&0&p+1&0\\
z&0&0&p+1
\end{bmatrix}\qquad
\begin{bmatrix}
0&0&0&0\\
x'&p+1&0&0\\
y'&0&p+1&0\\
z'&0&0&p+1
\end{bmatrix},
\]
respectively, we have $C^*(E)\otimes \KKK\simeq C^*(E')\otimes \KKK$ if and only if 
the number of entries in $(x,y,z)$ which are multiples of $p$ agrees 
with the number of entries in $(x',y',z')$ which are multiples of $p$.
\end{examp}

\begin{remar}
Note that $K\!K^1(\OO_{p+1}\oplus \OO_{p+1}\oplus \OO_{p+1},\KKK)$ does not tell the full story about stable isomorphism among the possible extensions fitting in \eqref{ourext}, as indeed, we have only 4 stable isomorphism classes among the $p^3$ different extensions. 
\end{remar}

\section{Acknowledgments}

This work was supported by the NordForsk Research Network ``Operator algebras and dynamics'' grant \# 11580 and by the Faroese Research Council.  The second and third named author wishes to thank the  Department of Mathematical Sciences at the University of Copenhagen for their support and hospitality.

\end{document}